\documentclass[11pt]{article}
\usepackage{amsmath}
\usepackage{amssymb}
\usepackage{amsthm}
\usepackage[usenames]{color}
\usepackage{amscd}
\usepackage{dsfont}
\usepackage[colorlinks=true,linkcolor=webgreen,filecolor=webbrown,
citecolor=webgreen]{hyperref}

\definecolor{webgreen}{rgb}{0,.5,0}
\definecolor{webbrown}{rgb}{.6,0,0}

\def\C{{\mathds{C}}}

\def\N{{\mathds{N}}}
\def\Z{{\mathds{Z}}}

\def\1{{\bf 1}}
\def\a{{\bf a}}

\def\id{\operatorname{id}}
\def\lcm{\operatorname{lcm}}

\newtheorem{theorem}{Theorem}
\newtheorem{corollary}{Corollary}

\newtheorem{lemma}{Lemma}

\begin{document}

\title{Menon's identity and arithmetical sums representing functions of several
variables}
\author{L\'aszl\'o T\'oth \thanks{The author gratefully acknowledges support from the
Austrian Science Fund (FWF) under the project Nr. P20847-N18.}\\ \\
Department of Mathematics, University of P\'ecs \\ Ifj\'us\'ag u. 6,
H-7624 P\'ecs, Hungary \\ and \\
Institute of Mathematics, Department of Integrative Biology \\
Universit\"at f\"ur Bodenkultur, Gregor Mendel-Stra{\ss}e 33, A-1180
Wien, Austria \\ \\ E-mail: \texttt{ltoth@gamma.ttk.pte.hu}}
\date{}
\maketitle

\begin{abstract} We generalize Menon's identity by considering sums representing arithmetical
functions of several variables. As an application, we give a formula
for the number of cyclic subgroups of the direct product of several
cyclic groups of arbitrary orders. We also point out extensions of
Menon's identity in the one variable case, which seems to not appear
in the literature.
\end{abstract}



\section{Introduction}

Menon's identity \cite{Kes1965} states, that for every $n\in \N:=
\{1,2,\ldots\}$,
\begin{equation} \label{Menon_id} \sum_{\substack{k=1\\
\gcd(k,n)=1}}^n \gcd(k-1,n)= \phi(n) \tau(n),
\end{equation}
where $\phi$ denotes Euler's function and $\tau(n)$ is the number of
divisors of $n$.

This identity has many generalizations derived by several authors.
For example, if $f$ is an arbitrary arithmetical function, then
\begin{equation} \label{SitaRamaiah_id}
\sum_{\substack{k=1\\ \gcd(k,n)=1}}^n f(\gcd(k-1,n))= \phi(n)
\sum_{d\mid n} \frac{(\mu*f)(d)}{\phi(d)} \quad (n\in \N),
\end{equation}
where $*$ stands for the Dirichlet convolution. Formula
\eqref{SitaRamaiah_id} was deduced, in an equivalent form, by Kesava
Menon \cite[Th.\ 1]{Kes1965} for $f$ multiplicative, and by Sita
Ramaiah \cite[Th.\ 9.1]{Sit1978} in a more general form.

Nageswara Rao \cite{Nag1972} proved that
\begin{equation} \label{Nageswara_id}
\sum_{\substack{k_1,\ldots,k_s=1\\ \gcd(k_1,\ldots, k_s,n)=1}}^n
\gcd(k_1-a_1,\ldots,k_s-a_s,n)^s = \phi_s(n) \tau(n) \quad (n\in
\N),
\end{equation}
where $a_1,\ldots,a_s\in \Z$, $\gcd(a_1,\ldots,a_s,n)=1$ and
$\phi_s(n)=n^s\prod_{p\mid n} (1-1/p^s)$ is the Jordan function of
order $s$.

Richards \cite{Ric1984} remarked that for any polynomial $g$ with
integer coefficients,
\begin{equation} \label{Richards_id}
\sum_{\substack{k=1\\ \gcd(k,n)=1}}^n \gcd(g(k),n)= \phi(n)
\sum_{d\mid n} \eta_g(d) \quad (n \in \N),
\end{equation}
where $\eta_g(d)$ stands for the number of solutions $x$ (mod $d$)
of the congruence $g(x)\equiv 0$ (mod $d$) such that $\gcd(x,d)=1$.
Haukkanen and Wang \cite{HauWan1996} gave a proof of formula
\eqref{Richards_id} in a more general setting.

In a recent paper Sury \cite{Sur2009} showed that
\begin{equation} \label{Sury_id}
\sum_{\substack{k_1,k_2,\ldots,k_r=1\\ \gcd(k_1,n)=1}}^n
\gcd(k_1-1,k_2,\ldots,k_r,n)= \phi(n) \sigma_{r-1}(n) \quad (n\in
\N),
\end{equation}
where $\sigma_k(n)=\sum_{d\mid n} d^k$.

Further generalizations of \eqref{Menon_id} and combinations of the
existing ones were given by Haukkanen \cite{Hau2005,Hau2008},
Haukkanen and McCarthy \cite{HauMcC1991}, Haukkanen and
Sivaramakrishnan \cite{HauSiv1994}, Sivaramakrishnan
\cite{Siv1974,Siv1979} and others. See also McCarthy \cite[Ch.\
1,2]{McC1986}. All of these identities represent functions of a
single variable.

Note that there are three main methods used in the literature to
prove Menon-type identities, namely: (i) group-theoretic method,
based on the Cauchy-Frobenius lemma, called also Burnside's lemma,
concerning group actions, see \cite{Kes1965,Ric1984,Sur2009}; (ii)
elementary number-theoretic methods based on properties of the
Dirichlet convolution and of multiplicative functions, see
\cite{Kes1965,Hau2005,HauWan1996, Sit1978}; (iii) number theoretic
method based on finite Fourier representations and Cauchy products
of $r$-even functions, cf. \cite{HauMcC1991,HauSiv1994,
McC1986,Nag1972}.

Recall the idea of the proof of \eqref{Menon_id} based on the
Cauchy-Frobenius lemma. Let $G$ be an arbitrary group of order $n$
and let $U_n:= \{k\in \N: 1\le k\le n, \gcd(k,n)=1\}$ be the group
of units (mod $n$). Consider the action of the group $U_n$ on $G$
given by $U_n \times G\ni (k,g) \mapsto g^k$. Here two elements of
$G$ belong to the same orbit iff they generate the same cyclic
subgroup. Hence the number of orbits is equal to the number of
cyclic subgroups of $G$, notation $c(G)$. We obtain, according to
the Cauchy-Frobenius lemma, that
\begin{equation} \label{Burnside_id}
c(G)=\frac1{\phi(n)} \sum_{\substack{k=1\\ \gcd(k,n)=1}}^n
\psi(G,k),
\end{equation}
where $\psi(G,k):= \# \{g\in G: g^k=g\}$ is the number of fixed
elements of $G$.

If $G=C_n$ is the cyclic group of order $n$, then $c(G)=\tau(n)$,
$\psi(G,k)=\gcd(k-1,n)$ and \eqref{Burnside_id} gives Menon's
identity \eqref{Menon_id}.

Now specialize \eqref{Burnside_id} to the case where $G$ is the
direct product of several cyclic groups of arbitrary orders, i.e.,
$G=C_{m_1} \times \cdots \times C_{m_r}$, where $m_1,\ldots, m_r\in
\N$ ($r\in \N$). We deduce that the number of its cyclic subgroups
is
\begin{equation} \label{form_cyclic_subgroups}
c(C_{m_1} \times \cdots \times C_{m_r}) = \frac1{\phi(q)} \sum_{\substack{k=1\\
\gcd(k, q)=1}}^q \gcd(k-1,m_1)\cdots \gcd(k-1,m_r),
\end{equation}
where $q=m_1\cdots m_r$.

Being motivated by this example and in order to evaluate the right
hand side of \eqref{form_cyclic_subgroups}, see Section
\ref{number_cyclic_subgroups}, we generalize in this paper Menon's
identity \eqref{Menon_id}, and also \eqref{SitaRamaiah_id} and
\eqref{Richards_id}, by considering arithmetical sums representing
functions of several variables. For example, using simple
number-theoretical arguments we derive the following identity:

Let $m_1,\ldots, m_r,M\in \N$ ($r\in \N$),
$m:=\lcm[m_1,\ldots,m_r]$, $m\mid M$ and $\a:=(a_1,\ldots,a_r)\in
\Z^r$. Then
\begin{equation*}
\frac1{\phi(M)} \sum_{\substack{k=1\\ \gcd(k,M)=1}}^M
\gcd(k-a_1,m_1)\cdots \gcd(k-a_r,m_r)
\end{equation*} \begin{equation} \label{my_general_Menon_id}
=  \sum_{d_1\mid m_1,\ldots,d_r\mid m_r} \frac{\phi(d_1) \cdots
\phi(d_r)}{\phi(\lcm[d_1,\ldots,d_r])} \eta^{(\a)}(d_1,\ldots, d_r),
\end{equation}
where
\begin{equation} \label{eta}
\eta^{(\a)}(d_1,\ldots,d_r)=
\begin{cases}
1, & \text{if} \ \gcd(d_i,a_i)=1 \ (1\le i\le r) \, \text{and}\,
\gcd(d_i,d_j)\mid a_i-a_j \ (1\le i,j\le r),\\ 0, &
\text{otherwise}.
\end{cases}
\end{equation}

Remark that \eqref{my_general_Menon_id} does not depend on $M$ and
it represents a multiplicative function of $r$ variables, to be
defined in Section \ref{section_prelim}. Also, each term of the sum
in the right hand side of \eqref{my_general_Menon_id} is an integer,
therefore the sum in the left hand side of
\eqref{my_general_Menon_id} is a multiple of $\phi(M)$ for any
$m_1,\ldots, m_r\in \N$.

If $r=2$ and $a_1=a_2=a\in \Z$ with $\gcd(a,m)=1$, then
\eqref{my_general_Menon_id} gives
\begin{equation} \label{my_general_Menon_id_r_2}
\sum_{\substack{k=1\\ \gcd(k,M)=1}}^M \gcd(k-a,m_1)\gcd(k-a,m_2) =
\phi(M) \sum_{d_1\mid m_1,d_2\mid m_2} \phi(\gcd(d_1,d_2)).
\end{equation}

If $m_1,\ldots, m_r$ are pairwise relatively prime, $M=m=m_1\cdots
m_r$ and $\a:=(a_1,\ldots,a_r)\in \Z^r$, then
\eqref{my_general_Menon_id} reduces to
\begin{equation} \label{rel_prime}
\sum_{\substack{k=1\\ \gcd(k,m)=1}}^m \gcd(k-a_1,m_1)\cdots
\gcd(k-a_r,m_r)= \phi(m) \tau(m_1,a_1)\cdots \tau(m_r,a_r),
\end{equation}
where $\tau(n,a)$ denotes the number of divisors $d$ of $n$ such
that $\gcd(d,a)=1$. Now, if $a_1=\ldots=a_r=a\in \Z$, then the right
hand side of \eqref{rel_prime} is $\phi(m) \tau(m,a)$.

Note that the arithmetical function of several variables
\begin{equation} \label{Pillai_multi}
A(m_1,\ldots,m_r):= \frac1{m} \sum_{k=1}^m \gcd(k,m_1)\cdots
\gcd(k,m_r),
\end{equation}
where $m_1,\ldots, m_r\in \N$ and $m:=\lcm[m_1,\ldots,m_r]$, as
above, was considered by Deitmar, Koyama and Kurokawa
\cite{DeiKoyKur2008} in case $m_j\mid m_{j+1}$ ($1\le j\le r-1$) by
studying analytic properties of some zeta functions of Igusa type.
The function \eqref{Pillai_multi} was investigated in
\cite{Tot2011}.

For $r=1$ and $m_1=m$ \eqref{Pillai_multi} reduces to the function
\begin{equation} \label{Pillai_A}
A(m):= \frac1{m} \sum_{k=1}^m \gcd(k,m) = \sum_{d\mid m}
\frac{\phi(d)}{d},
\end{equation}
of which arithmetical and analytical properties were surveyed in
\cite{Tot2010}.

We also generalize the function \eqref{Pillai_multi} and deduce
certain single variable extensions of Menon's identity, which seems
to not appear in the literature.

\section{Preliminaries} \label{section_prelim}

We present in this section some basic notions and properties to be
used in the paper.

We recall that an arithmetical function of $r$ variables is a
function $f:\N^r \to \C$, notation $f\in {\cal F}_r$.  If $f,g\in
{\cal F}_r$, then their convolution is defined as
\begin{equation} \label{convo_functions}
(f*g)(m_1,\ldots,m_r)= \sum_{d_1\mid m_1, \ldots, d_r\mid m_r}
f(d_1,\ldots,d_r) g(m_1/d_1, \ldots, m_r/d_r).
\end{equation}

A function $f\in {\cal F}_r$ is said to be multiplicative if it is
nonzero and
\begin{equation*} \label{def_mult}
f(m_1n_1,\ldots,m_rn_r)= f(m_1,\ldots,m_r) f(n_1,\ldots,n_r)
\end{equation*}
holds for any $m_1,\ldots,m_r,n_1,\ldots,n_r\in \N$ such that
$\gcd(m_1\cdots m_r,n_1\cdots n_r)=1$.

If $f$ is multiplicative, then it is determined by the values
$f(p^{a_1},\ldots,p^{a_r})$, where $p$ is prime and
$a_1,\ldots,a_r\in \N_0:=\{0,1,2,\ldots\}$. More exactly,
$f(1,\ldots,1)=1$ and for any $m_1,\ldots,m_r\in \N$,
\begin{equation*}
f(m_1,\ldots,m_r)= \prod_p f(p^{e_p(m_1)}, \ldots,p^{e_p(m_r)}),
\end{equation*}
where $m_i=\prod_p p^{e_p(m_i)}$ are the prime power factorizations
of $m_i$ ($1\le i\le r$), the products being over the primes $p$ and
all but a finite number of the exponents $e_p(m_i)$ being zero.

If $r=1$, i.e., in case of functions of a single variable we
reobtain the familiar notion of multiplicativity.

For example, the functions $(m_1,\ldots,m_r) \mapsto
\gcd(m_1,\ldots,m_r)$ and $(m_1,\ldots,m_r) \mapsto \linebreak
\lcm[m_1,\ldots,m_r]$ are multiplicative for any $r\in \N$.

The convolution \eqref{convo_functions} preserves the
multiplicativity of functions. This property, well-known in the one
variable case, follows easily from the definitions.

The product and the quotient of (nonvanishing) multiplicative
functions are multiplicative. Let $h\in {\cal F}_1$ and $f\in {\cal
F}_r$ be multiplicative functions. Then the functions
$(m_1,\ldots,m_r) \mapsto h(m_1)\cdots h(m_r)$ and $(m_1,\ldots,m_r)
\mapsto h(f(m_1,\ldots,m_r))$ are multiplicative. In particular, \\
$(m_1,\ldots,m_r) \mapsto h(\gcd(m_1,\ldots,m_r))$ and
$(m_1,\ldots,m_r) \mapsto h(\lcm[m_1,\ldots,m_r])$ are
multiplicative.

The definition and properties of multiplicativity for functions of
several variables go back to the work of Vaidyanathaswamy
\cite{Vai1931}.

In the one variable case $\1$, $\id$, $\id_t$ and $\phi_t$ ($t\in
\C$) will denote the functions given by $\1(n)=1$, $\id(n)=n$,
$\id_t(n)=n^t$ and $\phi_t(n)=n^t\prod_{p\mid n} \left(1-1/p^t
\right)$ ($n\in \N$), respectively.

Let $G=(g_1,\ldots,g_r)$ be a system of polynomials with integer
coefficients and consider the simultaneous congruences
\begin{equation} \label{sim_cong_m}
g_1(x)\equiv 0 \text{ (mod $m_1$)}, \ldots, g_r(x)\equiv 0 \text{
(mod $m_r$)}.
\end{equation}

Let $N_G(m_1,\ldots,m_r)$ denote the number of solutions $x$ (mod
$\lcm[m_1,\ldots,m_r]$) of \eqref{sim_cong_m}. Furthermore, let
$\eta_G(m_1,\ldots,m_r)$ denote the number of solutions $x$ (mod
$\lcm[m_1,\ldots,m_r]$) of  \eqref{sim_cong_m} such that
$\gcd(x,m_1)=1$, ..., $\gcd(x,m_r)=1$. These are other examples of
multiplicative functions of several variables, properties which
might be known, but we could not locate them in the literature. We
give their proof in Lemma \ref{lemma_sim_cong}.

If $r=1$, $m_1=m$ and $g_1=g$, then $N_G(m):=N_g(m)$ is the number
of solutions $x$ (mod $m$) of the congruence $g(x)\equiv 0$ (mod
$m$), which is multiplicative as a function of a single variable.
This is well-known, see e.g., \cite[Th.\ 5.28]{Apo1976}.

\begin{lemma} \label{lemma_sim_cong} For every system $G=(g_1,\ldots,g_r)$
of polynomials with integer coefficients the functions
$(m_1,\ldots,m_r) \mapsto N_G(m_1,\ldots,m_r)$ and $(m_1,\ldots,m_r)
\mapsto \eta_G(m_1,\ldots,m_r)$ are multiplicative.
\end{lemma}

\begin{proof} We prove the multiplicativity of the function $N_G$. In case of $\eta_G$
the proof is similar.

Let $m_1,\ldots,m_r,n_1,\ldots,n_r\in \N$ such that $\gcd(m_1\cdots
m_r,n_1\cdots n_r)=1$. Consider the simultaneous congruences
\eqref{sim_cong_m} together with
\begin{equation} \label{sim_cong_n}
g_1(x)\equiv 0 \text{ (mod $n_1$)}, \ldots, g_r(x)\equiv 0 \text{
(mod $n_r$)},
\end{equation}
\begin{equation} \label{sim_cong_mn}
g_1(x)\equiv 0 \text{ (mod $m_1n_1$)}, \ldots, g_r(x)\equiv 0 \text{
(mod $m_rn_r$)}.
\end{equation}

If $x$ is any solution of \eqref{sim_cong_mn}, then $x$ is a
solution of both \eqref{sim_cong_m} and \eqref{sim_cong_n}.

Conversely, assume that $x^*$ is a solution of \eqref{sim_cong_m}
and $x^{**}$ is a solution of  \eqref{sim_cong_n}. Consider the
simultaneous congruences
\begin{equation} \label{sim_cong_Chinese}
x\equiv x^* \text{ (mod $\lcm[m_1,\ldots,m_r]$)}, \ x\equiv x^{**}
\text{ (mod $\lcm[n_1,\ldots,n_r]$)}.
\end{equation}

Let $m:=\lcm[m_1,\ldots,m_r]$, $n:=\lcm[n_1,\ldots,n_r]$. By the
Chinese remainder theorem \eqref{sim_cong_Chinese} has a unique
solution $\tilde{x}$ (mod $mn$), where $mn=
\lcm[m_1n_1,\ldots,m_rn_r]$. Here $\tilde{x}$ is a solution of
\eqref{sim_cong_mn}, completing the proof.
\end{proof}

The following lemma is a known property, it follows easily by the
inclusion-exclusion principle, cf. \cite[Th.\ 5.32]{Apo1976}.

\begin{lemma} \label{lemma_phi} Let $n,d,x\in \N$ such that $d\mid n$, $1\le x\le d$,
$\gcd(x,d)=1$. Then
\begin{equation*}
\# \{k\in \N: 1\le k \le n, k\equiv x \ \text{\rm (mod $d$)},
\gcd(k,n)=1 \}=\phi(n)/\phi(d).
\end{equation*}
\end{lemma}

\section{Main results}

For $m_1,\ldots,m_r\in \N$ ($r\in \N$) let $m:
=\lcm[m_1,\ldots,m_r]$ and let $M\in \N$, $m\mid M$. Let
$F=(f_1,\ldots,f_r)$ be a system of arithmetical functions of one
variable and $G=(g_1,\ldots,g_r)$ be a system of polynomials with
integer coefficients.

Consider the arithmetical functions of $r$ variables
\begin{equation} \label{def_func_S}
S_F^{(G)}(m_1,\ldots,m_r):= \frac1{M} \sum_{k=1}^M
f_1(\gcd(g_1(k),m_1))\cdots f_r(\gcd(g_r(k),m_r)),
\end{equation}
\begin{equation} \label{def_func_R}
R_F^{(G)}(m_1,\ldots,m_r):= \frac1{\phi(M)}
\sum_{\substack{k=1\\\gcd(k,M)=1}}^M f_1(\gcd(g_1(k),m_1))\cdots
f_r(\gcd(g_r(k),m_r)).
\end{equation}

\begin{theorem} \label{theorem_S} If $F$ and $G$ are arbitrary systems of arithmetical
functions and polynomials with integer coefficients, respectively,
then for any $m_1,\ldots,m_r\in \N$,
\begin{equation} \label{repr_S}
S_F^{(G)}(m_1,\ldots,m_r) =  \sum_{d_1\mid m_1,\ldots,d_r\mid m_r}
\frac{(\mu*f_1)(d_1) \cdots (\mu*f_r)(d_r)} {\lcm[d_1,\ldots,d_r]}
N_G(d_1,\ldots,d_r),
\end{equation}
which does not depend on $M$.
\end{theorem}

\begin{proof} Writing $f_i=\1 * (\mu * f_i)$ ($1\le i\le r$) we
obtain
\begin{equation*}
S_F^{(G)}(m_1,\ldots,m_r)= \frac1{M} \sum_{k=1}^M \sum_{d_1\mid
\gcd(g_1(k),m_1)} (\mu*f_1)(d_1) \cdots \sum_{d_r\mid
\gcd(g_r(k),m_r)} (\mu*f_r)(d_r)
\end{equation*}
\begin{equation*}
=  \frac1{M} \sum_{d_1\mid m_1,\ldots,d_r\mid m_r} (\mu*f_1)(d_1)
\cdots (\mu*f_r)(d_r) \sum_{\substack{1\le k\le M
\\ g_1(k)\equiv 0 \text{ (mod $d_1$)},\ldots, g_r(k)\equiv 0 \text{ (mod $d_r$)}}}
1,
\end{equation*}
where the inner sum is $(M/\lcm[d_1,\ldots,d_r])
N_G(d_1,\ldots,d_r)$.
\end{proof}

\begin{corollary} If $F$ is a system of multiplicative arithmetical functions and $G$ is any system of
polynomials with integer coefficients, then the function
$(m_1,\ldots,m_r) \mapsto S_F^{(G)}(m_1,\ldots,m_r)$ is
multiplicative.
\end{corollary}

\begin{proof} By Theorem \ref{theorem_S} and Lemma \ref{lemma_sim_cong} the function $S_F^{(G)}$ is the
convolution of multiplicative functions, hence it is multiplicative.
\end{proof}

For the function $A(m_1,\ldots,m_r)$ given by \eqref{Pillai_multi}
we have the next representation.

\begin{corollary} {\rm (\cite[Prop.\ 12]{Tot2011}, $f_1=\ldots =f_r=\id$, $g_1(x)=\ldots =g_r(x)=x$)}
\begin{equation}
 \frac1{M} \sum_{k=1}^M \gcd(k,m_1)\cdots \gcd(k,m_r) =
 \sum_{d_1\mid m_1,\ldots,d_r\mid m_r} \frac{\phi(d_1) \cdots
\phi(d_r)} {\lcm[d_1,\ldots,d_r]},
\end{equation}
which is multiplicative.
\end{corollary}

For other special choices of $F$ and $G$ similar results can be
derived if the values $N_G(d_1,\ldots,d_r)$ are known, but we turn
our attention to the function $R_F^{(G)}(m_1,\ldots,m_r)$ defined by
\eqref{def_func_R}.

\begin{theorem} \label{theorem_R} If $F$ and $G$ are arbitrary systems
of arithmetical functions and polynomials with integer coefficients,
respectively, then for any $m_1,\ldots,m_r\in \N$,
\begin{equation} \label{repr_R}
R_F^{(G)}(m_1,\ldots,m_r) = \sum_{d_1\mid m_1,\ldots,d_r\mid m_r}
\frac{(\mu*f_1)(d_1) \cdots (\mu*f_r)(d_r)}
{\phi(\lcm[d_1,\ldots,d_r])} \eta_G(d_1,\ldots,d_r),
\end{equation}
which does not depend on $M$.
\end{theorem}

\begin{proof} Similar to the proof of Theorem \ref{theorem_S},
\begin{equation*}
R_F^{(G)}(m_1,\ldots,m_r)= \frac1{\phi(M)} \sum_{\substack{k=1\\
\gcd(k,M)=1}}^M \sum_{d_1\mid \gcd(g_1(k),m_1)} (\mu*f_1)(d_1)
\cdots \sum_{d_r\mid \gcd(g_r(k),m_r)} (\mu*f_r)(d_r)
\end{equation*}
\begin{equation*}
= \frac1{\phi(M)} \sum_{d_1\mid m_1,\ldots,d_r\mid m_r}
(\mu*f_1)(d_1) \cdots (\mu*f_r)(d_r) \sum_{\substack{1\le k\le M
\\ \gcd(k,M)=1\\ g_1(k)\equiv 0 \text{ (mod $d_1$)},\ldots, g_r(k)\equiv 0 \text{ (mod $d_r$)}}}
1,
\end{equation*}
where the inner sum is $(\phi(M)/\phi(\lcm[d_1,\ldots,d_r]))
\eta_G(d_1,\ldots,d_r)$ by Lemma \ref{lemma_phi}.
\end{proof}

In the one variable case ($r=1$) Theorem \ref{theorem_R} is a
special case of \cite[Theorem]{HauWan1996}, giving, with $f_1=f$,
$g_1=g$, $m_1=m$,
\begin{equation} \label{R_f_g}
R_f^{(g)}(m):= \frac1{\phi(m)} \sum_{\substack{k=1\\
\gcd(k,m)=1}}^m f(\gcd(g(k),m)) = \sum_{d\mid m}
\frac{(\mu*f)(d)}{\phi(d)}\eta_g(d),
\end{equation}
and for $f=\id$ this reduces to \eqref{Richards_id}.

\begin{corollary} Assume that $g_1=\ldots =g_r=g$ and $m_1,\ldots,m_r$ are pairwise
relatively prime. Then
\begin{equation} \label{case_pairw_rel_prime}
R_F^{(G)}(m_1,\ldots,m_r)= R_{f_1}^{(g)}(m_1) \cdots
R_{f_r}^{(g)}(m_r).
\end{equation}
\end{corollary}

\begin{proof} For any $d_1\mid m_1, \ldots, d_r\mid m_r$,
$\eta_G(\lcm[d_1,\ldots,d_r])=\eta_g(d_1\cdots d_r)=
\eta_g(d_1)\cdots \eta_g(d_r)$ and obtain from \eqref{repr_R} that
\begin{equation*}
R_F^{(G)}(m_1,\ldots,m_r) = \sum_{d_1\mid m_1}
\frac{(\mu*f_1)(d_1)}{\phi(d_1)}\eta_g(d_1) \cdots \sum_{d_r\mid
m_r} \frac{(\mu*f_r)(d_r)}{\phi(d_r)}\eta_g(d_r),
\end{equation*}
giving \eqref{case_pairw_rel_prime} using the notation of
\eqref{R_f_g}.
\end{proof}

\begin{corollary} If $F$ is a system of multiplicative arithmetical functions and $G$ is any system of
polynomials with integer coefficients, then the function
$(m_1,\ldots,m_r) \mapsto R_F^{(G)}(m_1,\ldots,m_r)$ is
multiplicative.
\end{corollary}

\begin{proof} By Theorem \ref{theorem_R} and Lemma
\ref{lemma_sim_cong} the function $R_F^{(G)}$ is the convolution of
multiplicative functions, hence it is multiplicative.
\end{proof}

In case of multiplicative functions $f_i$ ($1\le i\le r$) we can
assume that $m_i>1$ ($1\le i\le r$), since for $m_i=1$ the
corresponding factors of \eqref{def_func_R} are equal to $1$.

\begin{corollary} {\rm ($f_1=\id_{t_1},\ldots, f_r=\id_{t_r}$)}
\begin{equation*}
R_{t_1,\ldots,t_r}^{(G)}(m_1,\ldots,m_r):= \frac1{\phi(M)} \sum_{\substack{k=1\\
\gcd(k,M)=1}}^M (\gcd(g_1(k),m_1))^{t_1} \cdots
(\gcd(g_r(k),m_r))^{t_r}
\end{equation*}
\begin{equation} = \sum_{d_1\mid m_1,\ldots,d_r\mid m_r}
\frac{\phi_{t_1}(d_1) \cdots \phi_{t_r}(d_r)}
{\phi(\lcm[d_1,\ldots,d_r])} \eta_G(d_1,\ldots,d_r),
\end{equation}
representing a multiplicative function.
\end{corollary}

\begin{corollary} \label{cor_f_equal} {\rm ($f_1=\ldots =f_r=\id$)}
\begin{equation*}
R_r^{(G)}(m_1,\ldots,m_r):= \frac1{\phi(M)} \sum_{\substack{k=1\\
\gcd(k,M)=1}}^M \gcd(g_1(k),m_1) \cdots \gcd(g_r(k),m_r)
\end{equation*}
\begin{equation} = \sum_{d_1\mid m_1,\ldots,d_r\mid m_r}
\frac{\phi(d_1) \cdots \phi(d_r)} {\phi(\lcm[d_1,\ldots,d_r])}
\eta_G(d_1,\ldots,d_r),
\end{equation}
representing a (positive) integer valued multiplicative function.
\end{corollary}

\begin{proof} The function $(m_1,\ldots,m_r)\mapsto \phi(m_1) \cdots
\phi(m_r)/\phi(\lcm[m_1,\ldots,m_r])$ is multiplicative and its
values are integers, since $\phi(p^{e_1}) \cdots
\phi(p^{e_r})/\phi(\lcm[p^{e_1}, \ldots, p^{e_r}])$ are integers for
any prime $p$ and any $e_1,\ldots,e_r\in \N$.
\end{proof}

\begin{corollary} {\rm ($f_1=\id_{t_1},\ldots, f_r=\id_{t_r}$, $g_1(x)=x-a_1$,..., $g_r(x)=x-a_r$)}

For any $\a:=(a_1,\ldots,a_r)\in \Z^r$,
\begin{equation*}
R_{t_1,\ldots,t_r}^{(\a)}(m_1,\ldots,m_r):= \frac1{\phi(M)} \sum_{\substack{k=1\\
\gcd(k,M)=1}}^M (\gcd(k-a_1,m_1))^{t_1} \cdots
(\gcd(k-a_r,m_r))^{t_r}
\end{equation*}
\begin{equation} \label{general_Menon_id_t} = \sum_{d_1\mid m_1,\ldots,d_r\mid m_r}
\frac{\phi_{t_1}(d_1) \cdots
\phi_{t_r}(d_r)}{\phi(\lcm[d_1,\ldots,d_r])} \eta^{(\a)}(d_1,\ldots
d_r),
\end{equation}
where $\eta^{(\a)}(d_1,\ldots,d_r)$ is defined by \eqref{eta}.
\end{corollary}

\begin{proof} It is well-known that for $d_1,\ldots,d_r\in \N$ the simultaneous congruences
$x \equiv a_1$ (mod $d_1$), ..., $x \equiv a_r$ (mod $d_r$) admit
solutions iff $\gcd(d_i,d_j)\mid a_i-a_j$ ($1\le i,j\le r$) and in
this case there is a unique solution $\overline{x}$ (mod
$\lcm[d_1,\ldots,d_r]$). Here $\gcd(\overline{x},d_1)=
\gcd(a_1,d_1)$, ..., $\gcd(\overline{x},d_r)= \gcd(a_r,d_r)$ and
obtain for the values of $\eta^{(\a)}(d_1,\ldots,d_r)$ formula
\eqref{eta}.
\end{proof}

For $t_1=\ldots =t_r=1$ we obtain from \eqref{general_Menon_id_t}
formula \eqref{my_general_Menon_id} given in the Introduction.

\begin{corollary} {\rm ($r=2$, $f_1=f_2=\id$)}
\begin{equation*}
R_2^{(G)}(m_1,m_2):= \frac1{\phi(M)} \sum_{\substack{k=1\\
\gcd(k,M)=1}}^M \gcd(g_1(k),m_1) \gcd(g_2(k),m_2) \end{equation*}
\begin{equation} = \sum_{d_1\mid m_1, d_2\mid m_2}
\phi(\gcd(d_1,d_2)) \eta_G(d_1,d_2).
\end{equation}
\end{corollary}

\begin{proof} Use that $\phi(a)\phi(b)= \phi(\gcd(a,b))\phi(\lcm[a,b])$
for any $a,b\in \N$. Note that this holds for any multiplicative
function written instead of $\phi$.
\end{proof}

\begin{corollary} {\rm ($r=2$, $f_1=f_2=\id$, $g_1(x)=x-a_1$, $g_2(x)=x-a_2$, $a_1,a_2\in \Z$)}
\begin{equation*}
R_2^{(a_1,a_2)}(m_1,m_2):= \frac1{\phi(M)} \sum_{\substack{k=1\\
\gcd(k,M)=1}}^M \gcd(k-a_1,m_1) \gcd(k-a_2,m_2)
\end{equation*}
\begin{equation} \label{R_r_2_spec}  = \sum_{\substack{d_1\mid m_1, d_2\mid m_2\\
\gcd(d_1,a_1)=1, \gcd(d_2,a_2)=1\\ \gcd(d_1,d_2)\mid a_1-a_2}}
\phi(\gcd(d_1,d_2)).
\end{equation}
\end{corollary}

\begin{corollary} {\rm ($r=2$, $f_1=f_2=\id$, $g_1(x)=x-a_1$, $g_2(x)=x-a_2$, $|a_1-a_2|=1$)}

Let $a_1,a_2\in \Z$ with $|a_1-a_2|=1$. Then the multiplicative
function $R_2^{(a_1,a_2)}(m_1,m_2)$, given by \eqref{R_r_2_spec} can
be represented as
\begin{equation}
R_2^{(a_1,a_2)}(m_1,m_2) = \sum_{\substack{d_1\mid m_1, d_2\mid m_2\\
\gcd(d_1,a_1)=1, \gcd(d_2,a_2)=1 \\ \gcd(d_1,d_2)=1}} 1,
\end{equation} and for any prime $p$ and any $u,v\in \N$,
\begin{equation}
R_2^{(a_1,a_2)}(p^u,p^v) =
\begin{cases} u+v+1, & p\nmid a_1, p\nmid a_2,\\
u+1, & p\nmid a_1, p\mid a_2,\\
v+1, & p\mid a_1, p\nmid a_2,\\
1, & p\mid a_1, p\mid a_2.
\end{cases}
\end{equation}
\end{corollary}

Now we deduce formula \eqref{my_general_Menon_id_r_2} given in the
Introduction.

\begin{corollary} {\rm ($r=2$, $f_1=f_2=\id$, $g_1(x)=g_2(x)=x-a$, $a\in \Z$, $\gcd(a,m)=1$)}
\label{cor11}

If $\gcd(a,m)=1$ then
\begin{equation*}
R_2^{(a)}(m_1,m_2):= \frac1{\phi(M)} \sum_{\substack{k=1\\
\gcd(k,M)=1}}^M \gcd(k-a,m_1) \gcd(k-a,m_2)
\end{equation*}
\begin{equation} \label{R_r_2_speci_1}  = \sum_{d_1\mid m_1, d_2\mid m_2}
\phi(\gcd(d_1,d_2)).
\end{equation}
\end{corollary}

Note that other special systems of $F$ and $G$ can be considered
too. We give the following example.

\begin{corollary} {\rm ($r=2$, $f_1=f_2=\id$, $g_1(x)=g_2(x)=x^2-a$, $\gcd(a,m)=1$)}
  For every $m=\lcm[m_1,m_2]$ odd with $\gcd(a,m)=1$,
\begin{equation*}
\sum_{\substack{k=1\\
\gcd(k,m)=1}}^m \gcd(k^2-a,m_1) \gcd(k^2-a,m_2) \end{equation*}
\begin{equation}
= \phi(m) \sum_{d_1\mid m_1, d_2\mid m_2} \phi(\gcd(d_1,d_2))
\prod_{p\mid \lcm[d_1,d_2]} \left(1+\left(\frac{a}{p}\right)\right),
\end{equation}
which is a multiple of $\phi(m)$, where $\left( \frac{a}{p}\right)$
is the Legendre symbol.

In particular, for $a=1$ and every $m_1,m_2$ odd,
\begin{equation*}
\sum_{\substack{k=1\\
\gcd(k,m)=1}}^m \gcd(k^2-1,m_1) \gcd(k^2-1,m_2) \end{equation*}
\begin{equation}
= \phi(m) \sum_{d_1\mid m_1, d_2\mid m_2}
\phi(\gcd(d_1,d_2))2^{\omega(\lcm[d_1,d_2])},
\end{equation}
where $\omega(n)$ denotes the number of distinct prime factors of
$n$.
\end{corollary}

\begin{proof} The congruence $x^2\equiv a$ (mod $p^e$) has
$1+\left(\frac{a}{p}\right)$ solutions (mod $p^e$) for any prime
$p>2$, $p\nmid a$ and any $e\in \N$. Therefore,
$\eta_g(n)=\prod_{p\mid n} \left(1+\left(\frac{a}{p}\right)\right)$
for any $n\in \N$ odd with $\gcd(n,a)=1$. Apply Corollary
\ref{cor_f_equal}.
\end{proof}

\section{The number of cyclic subgroups of the direct product of several cyclic groups}
\label{number_cyclic_subgroups}

\begin{theorem} Let $m_1,\ldots, m_r\in \N$. The number of cyclic subgroups of the group
$C_{m_1}\times \cdots \times C_{m_r}$ is given by the formula
\begin{equation} \label{final_form_cyclic_subgroups}
c(C_{m_1}\times \cdots \times C_{m_r})= \sum_{d_1\mid
m_1,\ldots,d_r\mid m_r} \frac{\phi(d_1) \cdots
\phi(d_r)}{\phi(\lcm[d_1,\ldots,d_r])},
\end{equation}
representing a multiplicative function of $r$ variables.

In particular, the number of cyclic subgroups of $C_{m_1}\times
C_{m_2}$ is
\begin{equation} \label{final_form_cyclic_subgroups_2}
c(C_{m_1}\times C_{m_2})= \sum_{d_1\mid m_1, d_2\mid m_2}
\phi(\gcd(d_1,d_2)),
\end{equation}
and for any prime $p$ and any $u,v\in \N$ with $u\ge v$,
\begin{equation} \label{speci_1_values}
c(C_{p^u}\times C_{p^v}) = 2\left(1+p+p^2+\ldots
+p^{v-1}\right)+\left( u-v+1 \right)p^v.
\end{equation}
\end{theorem}

\begin{proof} Formula \eqref{final_form_cyclic_subgroups} follows at
once from \eqref{form_cyclic_subgroups}, deduced in the
Introduction, by applying \eqref{my_general_Menon_id} with
$a_1=\ldots =a_r=1$ and $M=m_1\cdots m_r$. For the case $r=2$ see
Corollary \ref{cor11}. The values of $c(C_{p^u}\times C_{p^v})$ are
easily obtained by \eqref{final_form_cyclic_subgroups_2}.
\end{proof}

Note that certain formulae for the number of cyclic subgroups of the
$p$-group $C_{p^{\alpha_1}} \times \cdots \times C_{p^{\alpha_r}}$
($\alpha_1,\ldots,\alpha_r\in \N$), including \eqref{speci_1_values}
were deduced in the recent paper \cite[Section 4]{Tar2010} by a
different method. Formulae \eqref{final_form_cyclic_subgroups} and
\eqref{final_form_cyclic_subgroups_2} are given, without proof, in
\cite{OEIS} in cases $r=3$, $m_1=m_2=m_3$ and $r=2$, $m_1=m_2$,
respectively.

\section{Further remarks}

In case $r=1$ formulae \eqref{repr_S} and \eqref{repr_R} can be used
to deduce new identities, representing functions of a single
variable, if the values $N_{g_1}(n)$, respectively $\eta_{g_1}(n)$
($n\in \N$) are known. As examples, we point out the next
identities.

\begin{corollary} Let $j\in \N$. For every $n\in \N$,
\begin{equation}
\frac1{n} \sum_{k=1}^n \gcd(k^j,n)= \sum_{d \mid n}
\frac{\phi(d)N^{(j)}(d)}{d},
\end{equation}
where the multiplicative function $N^{(j)}$ is given by
$N^{(j)}(p^a)=p^{[(j-1)a/j]}$ for every prime power $p^a$ ($a\in
\N$), $[y]$ denoting the greatest integer $\le y$.
\end{corollary}

\begin{proof} Apply formula \eqref{repr_S} for $r=1$, $f_1=\id$, $g(x)=x^j$
where the number of solutions of the congruence $x^j\equiv 0$ (mod
$p^a$) is $p^{[(j-1)a/j]}$, as it can be checked easily.
\end{proof}

In what follows consider formula \eqref{repr_R} for $r=1$, $f_1=\id$
and $g_1=g$. Then  \eqref{repr_R} reduces to \eqref{Richards_id}.

\begin{corollary} Let $a,b\in \Z$ with $\gcd(b,n)=1$. Then for every $n\in \N$,
\begin{equation} \label{Menon_id_linear} \sum_{\substack{k=1\\
\gcd(k,n)=1}}^n \gcd(bk-a,n)= \phi(n) \tau(n,a).
\end{equation}
\end{corollary}

\begin{proof} Apply formula \eqref{Richards_id} in case of the linear polynomial
$g(x)=bx-a$. Here $\eta_g(n)=1$ for $\gcd(a,n)=1$ and $\eta_g(n)=0$
otherwise.
\end{proof}

For $a=b=1$ \eqref{Menon_id_linear} reduces to \eqref{Menon_id}.

\begin{corollary} Let $n\in \N$. Then
\begin{equation} \label{Menon_id_quadr}
\sum_{\substack{k=1\\ \gcd(k,n)=1}}^n \gcd(k^2-1,n)= \phi(n) h(n),
\end{equation}
where
\begin{equation}
h(n)=\begin{cases} \tau(m^2), & n=m \ \text{odd}, \\
2\tau(m^2), & n=2m, m \ \text{odd},\\
4(\ell-1) \tau(m^2), & n=2^{\ell}m, \ell \ge 2, m \ \text{odd}.
\end{cases}
\end{equation}
\end{corollary}

\begin{proof} Apply formula \eqref{Richards_id} for the polynomial
$g(x)=x^2-1$. Any solution of $x^2\equiv 1$ (mod $n$) is coprime to
$n$, hence $\eta_g(n)=N_g(n)$. For the number $N_g(p^a)$ of
solutions of $x^2\equiv 1$ (mod $p^a$) it is well known that
$N_g(p^a)=2$ ($p$ odd prime, $a\in \N$), $N_g(2)=1$, $N_g(4)=2$,
$N_g(2^{\ell})=4$ ($\ell \ge 3$).
\end{proof}

Finally, let $j\in \N$ be fixed. Group the prime factors of $n\in
\N$ according to the values $\gcd(p-1,j)=d$ and write
$n=\prod_{d\mid j} n_d$, where for any $d\mid j$,
\begin{equation*}
n_d=\prod_{\substack{p^k\mid \mid n \\ \gcd(p-1,j)=d}} p^k.
\end{equation*}

\begin{corollary} For every $n\in \N$ odd,
\begin{equation} \label{Menon_id_power_j}
\sum_{\substack{k=1\\ \gcd(k,n)=1}}^n \gcd(k^j-1,n)= \phi(n)
\prod_{d\mid j} \tau(n^d_d).
\end{equation}
\end{corollary}

\begin{proof} For $g(x)=x^j-1$ we have $\eta_g(n)=N_g(n)$ with
$\eta(p^a)=\gcd(j,p-1)$ for every $p$ odd prime and $a\in \N$. Apply
formula \eqref{Richards_id}. Now, for $F(n):=\sum_{d\mid n}
\eta_g(d)$ one has $F(p^a)=1+a\gcd(j,p-1)=\tau(p^{a\gcd(j,p-1)})$
for any $p$ odd prime and $a\in \N$.
\end{proof}

\begin{corollary} {\rm ($j=6$)}
For every $n\in \N$ odd,
\begin{equation} \label{Menon_id_power_6}
\sum_{\substack{k=1\\ \gcd(k,n)=1}}^n \gcd(k^6-1,n)= \phi(n)
\tau(A^6)\tau(B^2),
\end{equation}
where $A$ is the product, with multiplicity, of the prime factors
$p\equiv 1$ (mod $6$) of $n$, and $B=n/A$.
\end{corollary}

\end{document}